\documentclass[a4paper,pdftex,reqno,11pt]{amsart}

\usepackage{epsfig}
\usepackage{subfigure}
\usepackage{calc}
\usepackage{amssymb}
\usepackage{amstext}
\usepackage{amsmath}
\usepackage{amsthm}
\usepackage{multicol}
\usepackage{pslatex}
\usepackage{apalike}
\usepackage[small]{caption}

\newtheorem{theorem}{Theorem}
\newtheorem{corollary}{Corollary}
\newtheorem{lemma}{Lemma}
\newtheorem{definition}{Definition}
\newtheorem{example}{Example}
\newtheorem{conjecture}{Conjecture}

\begin{document}

\title{Classes of Complete Simple Games that are All Weighted}

\author{Sascha Kurz$^\star$}                                                                                                                                                                
\address{$^\star$Department of Mathematics, University of Bayreuth, 95440 Bayreuth, Germany, sascha.kurz@uni-bayreuth.de}
\author{Nikolas Tautenhahn$^\dagger$}
\address{$^\dagger$LivingLogic AG, Markgrafenallee 44, 95448 Bayreuth, Germany}
\begin{abstract}
  Important decisions are likely made by groups of agents. Thus group decision making
is very common in practice. Very transparent group aggregating rules are given by weighted
voting, where each agent is assigned a weight. Here a proposal is accepted if the sum of the weights
of the supporting agents meets or exceeds a given quota. We study a more general class of binary voting
systems -- complete simple games -- and propose an algorithm to determine which sub classes,
parameterized by the agent's type composition, are weighted.
     
  \medskip
     
  
  \medskip
  
  \noindent
  \textbf{Copyright:} The present paper will be presented and published in a slightly different form at
  ICORES 2014, http://www.icores.org.  
\end{abstract}
\maketitle

\section{\uppercase{Introduction}}
\label{sec:introduction}

\noindent
Weighted voting is a method for group decision making. For simplicity, 
we assume that for each proposal on a certain issue the group members,
called agents for brevity,  options are either to vote {\lq\lq}yes{\rq\rq}
or {\lq\lq}no{\rq\rq}. The aggregated group decision then is also either 
{\lq\lq}yes{\rq\rq} or {\lq\lq}no{\rq\rq}. Those procedures are called
binary voting systems or games in the literature. The special case of a 
weighted game consists of a quota $q>0$ and weights $w_i\ge 0$ for
every participating  agent. With this, the aggregated group decision is 
{\lq\lq}yes{\rq\rq} if and only if the summed weights of the supporters
of a given proposal meets or exceeds the quota. For $n$ agents, such a game
is denoted by $[q;w_1,\dots,w_n]$.
   
Weighted voting systems are commonly applied whenever not all agents are
considered to be equal. Reasons may lie in heterogeneous competencies for
different issues, see e.g.\ \cite{grofman1983thirteen}. In stock corporations,
weights can arise as the number of shares that each shareholder owns, see
e.g.\ \cite{leech2013shareholder}. In two-tier voting systems like the US
Electoral College or the EU Council of Ministers, agents vote as a block
or represent countries with different population sizes, which then have to be
mapped to appropriate weights, see e.g.\ \cite{maaser2007equal}.   
  
Weighted games form a very concrete, compact, and well-studied sub class
of group decision rules. Nevertheless, some practical decision rules of
legislative bodies do not correspond to weighted games, like e.g.\ the present
rule of the EU Council of Ministers. In Section~\ref{sec:complete_games}, we
introduce and motivate the important super class of complete simple games.
So some complete simple games are weighted and others are not.

In principle, complete simple games (or simple games) can be very complicated.
Restricting to the subclass of symmetric (complete) simple games, where all
agents have equal capabilities, simplifies things dramatically, as first 
found out in \cite{may1952set}: All such games\footnote{More precisely,
May's Theorem applies to (complete) simple games with one type of agents 
(see Section~\ref{sec:complete_games}).} are weighted, 
i.e., have a relatively simple structure.

In this paper we aim to generalize May's Theorem by providing a strategy 
to classify all classes of complete simple games, according to the agent's 
type composition (see Definition~\ref{def_type_composition}),
with the property that every class member is weighted. It will turn out that
this can happen only if at most five different types of agents are present and
from all but one type there have to be very few agents, see lemmas 
\ref{lemma_many_types} and \ref{lemma_not_all_weighted}.

Exact formulas for the number of sub classes of weighted games are rather rare.
From May's Theorem, one can conclude that the number of weighted games with $n$
agents all of the same type\footnote{They all have the form 
$[q;1,\dots,1]$ with $q\in\{1,\dots,n\}$.} is given by $n$. In \cite{kurz2013dedekind},
the authors have presented an algorithm that can compute an exact enumeration 
formula for complete simple games with $t$ types of agents and $r$, so-called,
shift-minimal winning vectors depending on the number of agents $n$. No such algorithm
is known for weighted games. Having our classification result at hand, we can enumerate
the corresponding sub classes of weighted games since it will turn out, that in each case 
the number of occurring shift-minimal winning vectors is bounded by a small integer
and the restrictions from the agent's type composition can be easily incorporated  into
the enumeration algorithm. 

\section{\uppercase{Complete simple games}}
\label{sec:complete_games}

\noindent
A binary voting procedure can be modeled as a function $v:2^N\rightarrow\{0,1\}$ mapping
the coalition $S$ of supporting agents to the aggregated group decision $v(S)$, where
$N=\{1,\dots,n\}$ and $2^N$ denotes the set of subsets of $N$. Quite naturally, several
assumptions of a binary voting procedure are taken for granted: 
\begin{enumerate}
  \item[(1)] if no agent's supports the proposal, reject it;
  \item[(2)] if all agent's supports the proposal, accept it;
  \item[(3)] if the supporting clique of agents is enlarged by some additional agents, the
             group decision should not change from acceptance to rejection.
\end{enumerate}
More formally, we state:
\begin{definition}
  A pair $(v,N)$ is called simple game if $N$ is a finite set, $v:2^N\rightarrow\{0,1\}$
  satisfies $v(\emptyset)=0$, $v(N)=1$, and $v(S)\le v(T)$ for all $S\subseteq T\subseteq N$.
\end{definition}

A more demanding assumption is to require that the agents are linearly ordered according to
their capabilities to influence the final group decision. This can be formalized
with the desirability relation introduced in \cite{isbell1956class}.

\begin{definition}
  Let $(v,N)$ be a simple game. We write $i\sqsupset j$ (or $j \sqsubset i$) for two agents
  $i,j\in N$ if we have $v\Big(\{i\}\cup S\backslash\{j\}\Big)\ge v(S)$ for all 
  $\{j\}\subseteq S\subseteq N\backslash\{i\}$ and we abbreviate $i\sqsupset j$,
  $j\sqsupset i$ by $i\square j$.
\end{definition}

The relation $\square$ partitions the set of agents $N$ into equivalence classes $N_1,\dots,N_t$.
\begin{example}
  \label{ex_weighted_game}
  For the weighted game $[4;5,4,2,2,0]$ we have $N_1=\{1,2\}$, $N_2=\{3,4\}$, and $N_3=\{5\}$.
\end{example}
Agents having the same weight are contained in the same equivalence class, while the converse is
not necessarily true. But there always exists a different weighted representation of the same
game such that the agents of each equivalence class have the same weight. For 
Example~\ref{ex_weighted_game}, such a representation is e.g.\ given by $[2;2,2,1,1,0]$. 

\begin{definition}
  \label{def_complete_simple_game}
  A simple game $(v,N)$ is called complete if the binary relation $\sqsupset$ is a total
  preorder, i.e.,
  \begin{itemize}
    \item[(1)] $i\sqsupset i$ for all $i\in N$,
    \item[(2)] $i\sqsupset j$ or $j\sqsupset i$ for all $i,j\in N$, and
    \item[(3)] $i\sqsupset j$, $j\sqsupset h$ implies $i\sqsupset h$ for all $i,j,h\in N$.
  \end{itemize}
\end{definition}

All weighted games are obviously simple and complete.

\begin{definition}
  For a simple game $(v,N)$ a coalition $S\subseteq N$ is called winning if $v(S)=1$ and
  losing otherwise. If $v(S)=1$, $v(T)=0$ for all $T\subsetneq S$, then $S$ is called
  minimal winning. A coalition with $v(S)=0$, $v(T)=1$ for all
  $S\subsetneq T\subseteq N$ is called maximal losing.
\end{definition}

In Example~\ref{ex_weighted_game}, coalition $\{2,3\}$ is winning, $\{2\}$ is minimal
winning, $\{3\}$ is losing, and $\{3,5\}$ is maximal losing.

\begin{definition}
  \label{def_type_composition}
  For a complete simple game $(v,N)$, the vector $(n_1,\dots,n_t)\in\mathbb{N}_{>0}^t$, where
  $n_i=\left|N_i\right|$, is called type composition. The number $t$ of equivalence
  classes is called number of types (of agents).
\end{definition}

The type composition of Example~\ref{ex_weighted_game} is given by
$\begin{pmatrix}2,2,1\end{pmatrix}$ consisting of three types. Agents within the
same equivalence class are interchangeable, i.e., since coalition $\{2,3\}$ is
winning also the coalitions $\{1,3\}$, $\{1,4\}$, and $\{2,4\}$ have to be winning.
The combinatorial explosion of the set of corresponding winning coalitions can
be partially captured by:
\begin{definition}
  \label{def_winning_and_losing_vector}
  Let $(v,N)$ be a complete simple game with type composition $(n_1,\dots,n_t)$.
  Each vector $s=(s_1,\dots,s_t)\in\mathbb{N}^t$ with $0\le s_i\le n_i$ for all
  $1\le i\le t$ is called coalition vector of $(v,N)$. Coalition vector $s$
  is winning if we have $v(S)=1$ for coalitions $S\subseteq N$ with $\left|S\cap N_i\right|=s_i$
  for all $1\le i\le t$, and losing otherwise. 
\end{definition}  

The just mentioned four winning coalitions can be condensed to the winning vector $(1,1,0)$.

\begin{definition}
  For two vectors $a=(a_1,\dots,a_t)\in \mathbb{N}^t$ and $b=(b_1,\dots,b_t)\in\mathbb{N}^t$
  we write $a\le b$ if $a_i\le b_i$ for all $1\le i\le t$.
\end{definition}

If $a$ is a winning vector of a complete simple game and $a\le b$, then $b$ is winning too.
Next, we define a tightening of the concept of minimal winning and maximal losing coalitions
for coalition vectors. To this end we have to assume $1\sqsupset 2\sqsupset\dots\sqsupset n$
in the following.

\begin{definition}
  For two vectors $a=(a_1,\dots,a_t)\in \mathbb{N}^t$ and $b=(b_1,\dots,b_t)\in\mathbb{N}^t$
  we write $a\preceq b$ if $\sum_{j=1}^i a_j\le \sum_{j=1}^i b_j$ for all $1\le i\le t$.
  Vector $a$ is called shift-minimal winning (SMW), if $a$ is winning and all $b\preceq a$, with $b\neq a$, are losing. 
  Similarly, vector $a$ is called shift-maximal losing (SML), if $a$ is losing and all $b\succeq a$, with $b\neq a$, 
  are winning.
\end{definition}

An example is given by $(0,1,0)\preceq(1,0,0)$. The SMW vectors of 
Example~\ref{ex_weighted_game} are given by $(1,0,0)$ and $(0,2,0)$. The unique SML 
vector is given by $(0,1,1)$. We remark that each complete simple game is uniquely
characterized by its type composition and its full list of SMW vectors.
Of course, not every collection of coalition vectors for a given type composition is a feasible
set of SMW vectors.

\begin{definition}
  Let $a=(a_1,\dots,a_t)\in \mathbb{N}^t$ and $b=(b_1,\dots,b_t)\in\mathbb{N}^t$ be two vectors.
  We write $a\bowtie b$ if neither $a\preceq b$ nor $a\succeq b$, i.e., when they are incomparable.
  Mimicking the lexicographic order, we write $a\gtrdot b$ if there exists an index $k\in \{0,\dots,n-1\}$
  such that $a_j=b_j$ for all $1\le j\le k$ and $a_{j+1}>b_{j+1}$.
\end{definition} 

A parameterization theorem for complete simple games with $t$ types of agents has been given
in \cite{carreras1996complete}:
\begin{theorem}
  \label{thm_characterization_cs}

  \vspace*{0mm}

  \noindent
  \begin{itemize}
   \item[(a)] Let vector $\widehat{n}=(n_1,\dots,n_t)\in\mathbb{N}_{>0}^t$ and a matrix
              $$\mathcal{S}=\begin{pmatrix}s_{1,1}&s_{1,2}&\dots&s_{1,t}\\s_{2,1}&s_{2,2}&\dots&s_{2,t}\\
              \vdots&\ddots&\ddots&\vdots\\s_{r,1}&s_{r,2}&\dots&s_{r,t}\end{pmatrix}=
              \begin{pmatrix}\widehat{s}_1\\\widehat{s}_2\\\vdots\\\widehat{s}_r\end{pmatrix}$$
              satisfy the following properties:
              \begin{itemize}
               \item[(i)]   $0\le s_{i,j}\le n_j$, $s_{i,j}\in\mathbb{N}$ for $1\le i\le r$, $1\le j\le t$,
               \item[(ii)]  $\widehat{s}_i\bowtie\widehat{s}_j$ for all $1\le i<j\le r$,
               \item[(iii)] for each $1\le j<t$ there is at least one row-index $i$ such that
                            $s_{i,j}>0$, $s_{i,j+1}<n_{j+1}$ if $t>1$ and $s_{1,1}>0$ if $t=1$, and
               \item[(iv)]  $\widehat{s}_i\gtrdot \widehat{s}_{i+1}$ for $1\le i<r$.
              \end{itemize}
              Then, there exists a complete simple game $(v,N)$ associated to
              $\left(\widehat{n},\mathcal{S}\right)$.
   \item[(b)] Two complete simple games $\left(\widehat{n}_1,\mathcal{S}_1\right)$ and
              $\left(\widehat{n}_2,\mathcal{S}_2\right)$ are isomorphic if and only if 
              $\widehat{n}_1=\widehat{n}_2$ and $\mathcal{S}_1=\mathcal{S}_2$.
  \end{itemize}
\end{theorem}

Besides being rather technical, there is some easy interpretation for the stated conditions. 
Condition~(i) simply states that the $\widehat{s}_i$ are feasible with respect to the type composition
$\widehat{n}$. If we would not have $\widehat{s}_i\bowtie\widehat{s}_j$, then either 
$\widehat{s}_i\preceq\widehat{s}_j$ or $\widehat{s}_i\succeq\widehat{s}_j$, so that one of both
vectors can not be shift-minimal. Condition~(iii) is necessary to enforce equivalence classes
according to $\widehat{n}$ and condition~(iv) prevents from row permutations. We call two complete
simple games isomorphic if there exists a bijection for the respective agent's names preserving
winning and losing coalitions.

\begin{definition}
  A simple game $(v,N)$ is called weighted if and only if there exist weights $w_i\in\mathbb{R}_{\ge 0}$, 
  for all $i\in N$, and a quota $q\in\mathbb{R}_{>0}$ such that $v(S)=1$ is equivalent
  to $\sum_{i\in S} w_i\ge q$ for all $S\subseteq N$. 
\end{definition}

\section{\uppercase{(Non-) Weightedness}}
\label{sec:weightedness}

\noindent
We have mentioned in the introduction that some complete simple games are weighted while others
are not. In this section, we want to provide a method to decide which case occurs\footnote{Several
algorithms to decide whether a given simple game is weighted or not are known in the literature,
see e.g.\ \cite{0943.91005} for an overview.}.
\begin{example}
  \label{ex_csg}
  Let $\widehat{n}=(2,4)$ and $\mathcal{S}=\begin{pmatrix}2&0\\0&4\end{pmatrix}$, 
  then the complete simple game $\left(\widehat{n},\mathcal{S}\right)$ is not weighted.
\end{example}

Similar to the matrix $\mathcal{S}$ of the shift-minimal vectors, one can write down
a matrix $\mathcal{L}$ of the shift-maximal losing vectors. If running time is not an 
issue, this can be easily done algorithmically:

\begin{small}
\begin{verbatim}
 For each coalition vector a=(a_1,...a_t)
    determine whether a is winning or losing
 End
 For each losing vector a=(a_1,...a_t)
   ok=True
   For i from 1 to t
       If a_i<n_i and a+e_i is losing
       Then ok=False
   End
   If ok==True Then output a
 End
\end{verbatim}
\end{small}

Here $e_i$ denotes the $i$th unit vector and we have $\mathcal{L}=\begin{pmatrix}1&2\end{pmatrix}$
in Example~\ref{ex_csg}.

\begin{lemma}
  \label{lemma_non_weighted_certificate}
  For a complete simple game $(v,N)$ let $\tilde{S}$ be a matrix of (some) winning vectors
  and $\tilde{L}$ be a matrix of (some) losing vectors. If there exist (row) vectors  $x,y$ with
  non-negative real entries, $\Vert x\Vert_1=\Vert y\Vert_1>0$ and $x\tilde{S}\le y\tilde{L}$,
  then $(v,N)$ can not be weighted.
\end{lemma}
\begin{proof}
  Combining the fact that the weight of each winning coalition is larger than the weight of each
  losing coalitions with the weighting of coalitions induced by $x$ and $y$ gives a contradiction.
\end{proof}

A well known fact from the literature is the inverse statement, i.e., for each non-weighted (complete)
simple game there exists a set of winning coalitions (or winning vectors) and a set of losing coalitions
(or losing vectors) with multipliers $x,y$ certifying non-weightedness. The underlying concepts are
trading transforms, see \cite{0943.91005}, or dual multipliers in the theory of linear programming.
An example of a non-weighted complete simple game can be extended to other type compositions:

\begin{lemma}
  \label{lemma_non_weighted_extension}
  Let $G_1=(v,N)$ be a complete simple game with type composition $\widehat{n}=(n_1,\dots,n_t)$ and
  $\tilde{S}$ a matrix of winning vectors, $\tilde{L}$ a matrix of losing vectors, and $x,y$ be
  vectors according to Lemma~\ref{lemma_non_weighted_certificate}, which certify non-weightedness
  of $(v,N)$. For each vector $\widehat{m}\ge \widehat{n}$ there exists a non-weighted complete simple game
  $G_2=(v',N')$ with type composition $\widehat{m}$.
\end{lemma}
\begin{proof}
  We choose $N'$ such that $N\subseteq N'$ and set $v'(S)=v(S)$ for all $S\subseteq N$, i.e., winning
  vectors of $G_1$ are also winning in $G_2$ and losing vectors of $G_1$ are also losing in $G_2$. All
  coalition vectors of $G_2$ that are comparable to the already assigned vectors, i.e., to the coalition
  vectors of $G_1$, are accordingly set to be either winning or losing. For the remaining vectors
  we have some freedom, but for simplicity determine them to be losing vectors. We can easily
  check that $G_2$ is completely characterized and is indeed a complete simple game. Since
  the rows of $\tilde{S}$ are also winning vectors in $G_2$ and the rows of $\mathcal{L}$ are also losing
  vectors in $G_2$, we can apply Lemma~\ref{lemma_non_weighted_certificate} with the original
  vectors $x,y$ to deduce that $G_2$ is non-weighted. 
\end{proof}

As an example, let $(v,N)$ be uniquely characterized by $\widehat{n}=(3,4)$ and 
$\mathcal{S}=\begin{pmatrix}2&2\end{pmatrix}$. The matrix of shift-maximal losing vectors
is given by $\mathcal{L}=\begin{pmatrix}3&0\\1&4\end{pmatrix}$ and we have $x=(2)$, $y=(1,1)$
as a certificate for non-weightedness. For $\widehat{m}=(6,6)$ the construction of 
Lemma~\ref{lemma_non_weighted_extension} gives the game with type composition
$\widehat{m}$ and $\mathcal{S}=\begin{pmatrix}2&2\end{pmatrix}$. Now the matrix of shift-maximal
losing vectors is given by 
$\mathcal{L}=\begin{pmatrix}3&1&0\\0&4&6\end{pmatrix}^T$.
From the
reused vectors $x$, $y$ we can conclude the non-weightedness of the larger complete simple game.
The degree of freedom in the proof of Lemma~\ref{lemma_non_weighted_extension} allows
us to also conclude that the complete game given by type composition $\widehat{m}$ and
$\mathcal{S}=\begin{pmatrix}2&2\\0&6\end{pmatrix}$ is also non-weighted. Here we have 
$\mathcal{L}=\begin{pmatrix}3&0\\1&4\end{pmatrix}$. The common parts $\tilde{S}$ and
$\tilde{L}$ have to be chosen accordingly.

\begin{definition}
  \label{def_weighted_composition_type}
  We call a type composition $\widehat{n}=(n_1,\dots,n_t)$ weighted if all complete simple games,
  given by a matrix $\mathcal{S}$ of its SMW vectors and $\widehat{n}$, are weighted.
  Otherwise we call $\widehat{n}$ non-weighted. 
\end{definition}

Since $(3,4)\ge (2,4)$ we do not learn anything new, i.e., 
Lemma~\ref{lemma_non_weighted_extension} alone is sufficient to prove:
\begin{lemma}
  \label{lemma_not_all_weighted_special}
  Each type composition $\widehat{n}=(n_1,n_2)$ with $n_1\ge 2$ and $n_2\ge 4$ is non-weighted 
\end{lemma} 

\begin{lemma}
  \label{lemma_many_types}
  Each type composition $\widehat{n}=(n_1,\dots,n_t)\in\mathbb{N}_{>0}^t$ with
  $t\ge 6$ is non-weighted. 
\end{lemma}
\begin{proof}
  The game $(\widehat{m},\mathcal{S})$ with
  $$
    \mathcal{S}=\begin{pmatrix}
      1 & 1 & 0 & 0 & 0 & 0\\
      1 & 0 & 1 & 0 & 0 & 1\\
      1 & 0 & 0 & 1 & 1 & 0\\
      0 & 1 & 1 & 1 & 0 & 0\\
      0 & 0 & 1 & 1 & 1 & 1 
    \end{pmatrix}
  $$
  is complete, non-weighted, and has $\widehat{m}=(1,1,1,1,1,1)$ as its type composition with $6$ types.
  For $t>6$ we consider the complete simple non-weighted game with type composition with 
  $\widehat{m}=(1,\dots,1)\in\mathbb{N}_{>0}^t$ uniquely characterized by its matrix
  $$
    \mathcal{S}=\left(\begin{array}{cccccccccccc}
      1 & 1 & 0 & 0 & 0 & 0 & | & 1 & 0 & 1 & 0 & \dots \\
      1 & 0 & 1 & 0 & 0 & 1 & | & 0 & 1 & 0 & 1 & \dots \\
      1 & 0 & 0 & 1 & 1 & 0 & | & 1 & 0 & 1 & 0 & \dots \\
      0 & 1 & 1 & 1 & 0 & 0 & | & 0 & 1 & 0 & 1 & \dots \\
      0 & 0 & 1 & 1 & 1 & 1 & | & 1 & 0 & 1 & 0 & \dots 
    \end{array}\right)
  $$
  of SMW vectors. Lemma~\ref{lemma_non_weighted_extension} transfers
  the result to arbitrary type compositions $\widehat{n}$ with $t\ge 6$ types.
\end{proof}

With the help of lemmas~\ref{lemma_non_weighted_extension} and \ref{lemma_many_types},
we can propose the following strategy to classify all weighted type compositions $\widehat{n}$.
For small $n$, determine all complete simple games with at most $5$ types of agents and determine 
which ones are weighted or non-weighted. Taking only the smallest examples, we obtain a
generalized version of Lemma~\ref{lemma_not_all_weighted_special} and may hope
that all other cases correspond to weighted type compositions. Doing that we obtain:
  
\begin{lemma}
  \label{lemma_not_all_weighted}
  For each type composition $\widehat{n}=(2,4)$, $(2,2,2),(1,1,5),(1,2,3),(1,3,2)$,
  $(2,1,4),(2,4,1),(1,1,1,3),(1,1,3,1),(2,1,2,1),(2,3,1,1),(1,2,2,1),(1,2,1,2)$,\,\,\, $(1,1,2,2),
  (1,2,1,1,1),(1,1,2,1,1)$, $(1,1,1,2,1),(1,1,1,1,2)$ there exists a non-weighted complete
  simple game attaining $\widehat{n}$. 
\end{lemma}

\begin{conjecture}
  \label{main_conj}
  Each type composition $\widehat{n}$ is either weighted or there exists a type composition $\widehat{m}$,
  contained in the list of Lemma~\ref{lemma_not_all_weighted}, with $\widehat{n}\ge\widehat{m}$.
\end{conjecture}

In the next section, we propose an algorithmic approach to prove Conjecture~\ref{main_conj}.
As an example, we prove some special cases in Subsection~\ref{subsec:example_for_algorithm}.

\section{\uppercase{Weighted type compositions}}
\label{sec:weighted_composition_types}

\noindent
Due to Lemma~\ref{lemma_not_all_weighted_special}, for $t=2$ types of agents, the only possible
candidates for weighted type compositions are of the form $(1,\star)$, $(\star,1)$, $(\star,2)$,
and $(\star,3)$, where $\star$ stands for an arbitrary positive integer. 

\begin{definition}
  A set $\omega=(n_1,\dots,n_{i-1},\star,n_{i+1},\dots,n_t)$ of type compositions is called $i\star$ 
  family with $t$ types. We call $\omega$ weighted if all of its elements are weighted.
\end{definition}

\begin{corollary}(of Conjecture~\ref{main_conj})
 \label{main_cor}
 $(\star)$, $(1,\star)$, $(\star,1)$, $(\star,2)$, $(\star,3)$, $(\star,1,1)$, $(\star,1,2)$, $(\star,1,3)$,
 $(\star,2,1)$, $(\star,3,1)$, $(1,\star,1)$, $(1,1,4)$, $(1,2,2)$, $(\star,1,1,1)$, $(\star,1,1,2)$,
 $(\star,2,1,1)$, $(1,\star,1,1)$, $(1,1,2,1)$, and $(\star,1,1,1,1)$ are weighted. 
\end{corollary}

In this section, we propose an algorithm capable to prove that the type compositions of a given
$i\star$ family $\omega$ with $t$ types are weighted (if true). 

\subsection{Step 1: SMW Vectors}
\label{subsec:step1}
Consider a complete simple game with type composition $\widehat{n}\in\omega$ and matrix $\mathcal{S}$ of
its SMW vectors. We aim to write down a finite set of parameterized candidates for the rows
of $\mathcal{S}$ only depending on $\omega$. With $\alpha=\left\{a=(a_1,\dots,a_{i-1})\mid
0\le a_j\le n_j,\, 1\le j\le i-1\right\}$ we introduce the injective function 
$\tau:\alpha\rightarrow \mathbb{N}$ by $\tau(a)=\sum_{j=1}^{i-1} a_i\prod_{k=j+1}^{i-1} (n_k+1)$, i.e., we are
just numbering the elements of $\alpha$ in a convenient way. Additionally we set
$\tau(\omega):=\tau\big((n_1,\dots,n_{i-1})\big)+1=|\alpha|$. 

\begin{lemma}
  \label{lemma_shift_minimal_winning_candidates}
  Given an $i\star$ family $\omega$ with $t$ types let $C(\omega)$
  $$
    =\left\{\left(a,m_{\tau(a)}-c,b\right)\mid a\in \alpha, b\in \beta, c\in\mathbb{N},c\le\Lambda \right\},
  $$
  where $\beta=\left\{\left(b_{i+1},\dots,b_n\right)\in\mathbb{N}^{n-i}\mid
  b_h\le n_h\right\}$, $\Lambda=\sum_{h=i+1}^n n_h$, and the $m_j$ are free variables.
  For each simple game with type composition $\widehat{n}\in\omega$ and matrix 
  $\mathcal{S}=(\tilde{s}_1,\dots,\tilde{s}_r)^T$ of its SMW vectors,
  there exists an allocation of the $m_j$, such that $\tilde{s}_h\in C(\omega)$ for all
  $1\le h\le r$.
\end{lemma}
\begin{proof}
  Given a complete simple game with type composition $\widehat{n}\in\omega$ and matrix 
  $\mathcal{S}=(\tilde{s}_1,\dots,\tilde{s}_r)^T$ of its SMW vectors.
  For each $a\in\alpha$ let $\tilde{t}_a=(a,m,b)$, where $b=(b_{i+1},\dots,b_n)\in\beta$,
  the row-vector of $\mathcal{S}$ whose first $i-1$ coordinates coincide with $a$ and
  whose $i$th coordinate $m\in\mathbb{N}$ is maximal. If $\tilde{t}_a$ exists, we set 
  $m_{\tau(a)}=m$ and $m_{\tau(a)}=-1$ otherwise. Now let $\tilde{s}_j=(a,m',b')$ be an
  arbitrary SMW vector whose first $i-1$ components coincide with $a$.
  Clearly $m'\in\mathbb{N}$, $m'\le m$, and $b'=(b'_{i+1},\dots,b'_n)\in \beta$.   
  If $m'<m$ then there exists an index $k\in \{i+1,\dots,n\}$ with
  $m'+\sum_{h=i+1}^k b'_h>m+\sum_{h=i+1}^k b_h$ since $\sum_{h=1}^{i-1}a_h+m'<\sum_{h=1}^{i-1}a_h+m$
  and $\tilde{t}_a\bowtie \tilde{s}_j$. With $\sum_{h=i+1}^k b_h\ge 0$ and
  $\sum_{h=i+1}^k b'_h\le\sum_{h=i+1}^n n_h$ we have $m'> n-\Lambda$.  
\end{proof}

Thus we can parameterize the potential SMW vectors of a complete simple game
attaining $\omega$ using at most $\tau(\omega)$ parameters as elements in $C(\omega)$,
where 
\begin{equation}
  |C(\omega)|\le \left(\sum_{j=1,j\neq i}^n n_j\right)\cdot\prod_{j=1,j\neq i}^n (n_j+1), 
\end{equation}
i.e., the number rows $r$ is bounded by $\omega$.

\subsection{Step 2: Matrices of All Shift-minimal Winning Vectors}
\label{subsec:step2}

Given an $i\star$ family $\omega$ with $t$ types, the sets of SMW vectors 
of a complete simple game with type composition $\widehat{n}\in\omega$ are subsets of $C(\omega)$.
Thus, we can loop over all elements of $2^{C(\omega)}$ and need to check whether the selected subsets
satisfy the conditions of Theorem~\ref{thm_characterization_cs}(a). The technical difficulty we have
to face here is, that the entries can linearly depend on the parameters $m_j$. In \cite{kurz2013dedekind}
the similar situation, where all entries $s_{i,j}$ are parameters, has been treated. There it is shown
that all feasible cases, meeting the conditions of Theorem~\ref{thm_characterization_cs}(a), can be
formulated as a union of systems of linear inequality systems in terms of the parameters. Exemplarily, 
condition (a)(ii) is satisfied for a pair of indices $1\le i<j\le r$, if two further indices
$1\le h,k\le t$ exist with
$$
  \sum_{u=1}^h s_{i,u}+1\le \sum_{u=1}^hs_{j,u}
  \,\,\text{ and }\,\,
  \sum_{u=1}^k s_{i,u}\ge 1+ \sum_{u=1}^k s_{j,u}.
$$ 
Performing these steps yields a finite list $\mathcal{S}_1,\mathcal{S}_2,\dots$ of matrices, whose
entries are linear functions of the parameters $m_j$, such that the rows of each matrix $\mathcal{S}_h$
are the (parametric) SMW vectors of complete simple games with type composition in $\omega$
whenever the parameters $m_i$ satisfy the linear constraints of the corresponding polytope $P_h$. Moreover,
all complete simple games with type composition in $\omega$ are captured by one of the pairs $\mathcal{S}_h$,
$P_h$. (It is indeed possible to obtain a partition of the desired space.)  

\subsection{Step 3: Losing Vectors}
\label{subsec:step3}
For simplicity, we assume that we are given a single pair $\left(\mathcal{S}_h,P_h\right)$ according to 
Subsection~\ref{subsec:step2}. In \cite{kurz2013dedekind} the parametric Barvinok algorithm was applied
to count the respective number of complete simple games. Here we want to study weightedness so that we also
need a description (not necessarily the most compact description) of the set of losing vectors. To this end,
we mention that the SML vectors of a complete simple game are either incomparable to
all SMW vectors, and so contained in $C(\omega)$, or arise as so-called shifts of one of
the SMW vectors, i.e., special vectors that have a fairly small $\Vert\cdot\Vert_1$-distance
to one of the SMW vectors. Due to space limitations we just mention, that it is possible
to exactly describe a set of all candidates for SML vectors, similar as $C(\omega)$ for
the set of SMW vectors. Then we can again consider subsets of the set of candidates and
have to check that the implications, with respect to be a winning or a losing vector, are non-contradicting
and that the state of each vector can be deduced in any case. If properly implemented with all technical details,
things boil down to a splitting of a sub case $(\mathcal{S}_h,P_h)$ into a finite list of sub sub cases 
$$
  \left(\mathcal{S}_{h},\mathcal{L}_{h,1},P_{h,1}\right),
  \left(\mathcal{S}_{h},\mathcal{L}_{h,2},P_{h,2}\right),\dots,
$$  
where the rows of the $\mathcal{L}_{h,j}$ correspond to (not necessarily shift-maximal) losing vectors
and the $P_{h,j}$ are sub polytopes of $P_h$.    

\subsection{Step 4: Weighted Representation}
\label{subsec:step4}
For simplicity, we assume that we are given a single triple $\Gamma=\big(\mathcal{S}_h,\mathcal{L}_h,P_h\big)$
according to Subsection~\ref{subsec:step3}. For each integral choice of the parameters $m_j$ in $P_h$, we have
a unique complete simple game at hand and can check whether it is weighted with the help of a linear program,
see e.g.\ \cite{0943.91005}. If at least one of such games is non-weighted, then we can use the methods of 
Section~\ref{sec:weightedness} to deduce that a certain class of type compositions is non-weighted. So let
us assume that all games corresponding to $\Gamma$ are indeed weighted.

Thus, for each (of the possibly infinitely many) complete simple games corresponding to $\Gamma$, there exist
feasible weights obtained at a basis solution of the corresponding linear program, which is uniquely
determined by $\Gamma$ but depends on the parameters $m_i$. Nevertheless, the number of possible basis solutions
is finite. Thus, we can loop over all possible parametric basis solutions $\delta_j$ and determine the
corresponding list of polytopes $P_{h,k}$, such that $\delta_j$ yields a feasible weighting of the complete
simple games corresponding to $\mathcal{S}_h,\mathcal{L}_h$ whenever the parameters $m_l$ are in $P_{h,k}$.
If $\cup_k P_{h,k} =P_h$, then $\omega$ is weighted for sub case $\Gamma$.  

\subsection{Examples}
\label{subsec:example_for_algorithm}
In the previous four subsections we have sketched an algorithm that is capable to prove that a given $i\star$
family $\omega$ with $t$ types is weighted (if the statement is indeed true). Due to space limitations, we have
not given all technical, sometimes non-trivially, details. Instead, we want to give examples for special cases.

\begin{lemma}
  $\omega=(\star)$ is weighted.
\end{lemma}
\begin{proof}
  According to Step 1 we have $\tau(\omega)=1$ parameter $m_0$, $\Lambda=0$, and $C(\omega)=\{(m_0)\}$ with
  $|C(\omega)|=1$. Since each complete simple game consists of a least one shift-minimal winning coalition
  we obtain the one-element list $\Big(\mathcal{S}_1=(m_0),P_1=\{(m_0)\in\mathbb{R}^1\mid 1\le m_0\le n_1\}\Big)$.
  Since all elements of $C(\omega)$ have already been assigned to be shift-minimal winning, there remains
  the unique shifted vector $(m_0-1)$ to be losing in Step~3. In Step~4, we can obtain the basis solution
  $q=m_0$, $w_1=1$, which is feasible for the entire polytope $P_1$. (We only state the weights for each type of
  agents, numbered from $1$ to $t$.)
\end{proof}

\begin{lemma}
  $\omega=(1,\star)$ is weighted.
\end{lemma}
\begin{proof}
  According to Step~1 we have $\tau(\omega)=2$ parameters $m_0,m_1$, $\Lambda=0$, and $C(\omega)=\{(0,m_0),(1,m_1)\}$
  with $|C(\omega)|=2$. In Step~2 we have to consider the three non-empty subsets of $C(\omega)$. For the
  case $\mathcal{S}=\begin{pmatrix}0&m_0\end{pmatrix}$ we observe that condition (a)(iii) of 
  Theorem~\ref{thm_characterization_cs} can not be met for any allocation of the parameters $m_0,m_1$. Thus, there
  remain only two cases with non-empty polytopes for the parameters:
  $$
    \mathcal{S}_1=\begin{pmatrix}1&m_1\end{pmatrix}, 
    P_1=\left\{(m_1)\in\mathbb{R}\mid 0\le m_1\le n_2-1\right\}
  $$  
  and
  $$
    \mathcal{S}_2=\begin{pmatrix}1&m_1\\0&m_0\end{pmatrix},
    P_2=\left\{\!\!\left(\!\!\!\!\!\begin{array}{c}m_0\\m_1\end{array}\!\!\!\!\!\right)\!\in\!\mathbb{R}^2\!\mid\!\!\!\! 
    \begin{array}{l}m_1\ge 0\\ m_1+2\le m_0\\m_0\le n_2
    \end{array}\!\!\!\!\!\right\}
    .
  $$
  In Step~3, each of the two sub cases is split into two sub sub cases 
  $\Gamma_1=(\mathcal{S}_1,\mathcal{L}_{1,1},P_{1,1})$,
  $\Gamma_2=(\mathcal{S}_1,\mathcal{L}_{1,2},P_{1,2})$,
  $\Gamma_3=(\mathcal{S}_2,\mathcal{L}_{2,1},P_{2,1})$,
  $\Gamma_3=(\mathcal{S}_2,\mathcal{L}_{2,2},P_{2,2})$,
  with
  $\mathcal{L}_{1,1}=\begin{pmatrix}1&m_1-1\\0&n_2\end{pmatrix}$,
  $\mathcal{L}_{1,2}=\begin{pmatrix}0&n_2\end{pmatrix}$,
  $\mathcal{L}_{2,1}=\begin{pmatrix}1&m_1-1\\0&m_0-1\end{pmatrix}$,
  $\mathcal{L}_{2,2}=\begin{pmatrix}0&m_0-1\end{pmatrix}$,
  \begin{eqnarray*}
    P_{1,1}&=&\left\{\!(m_1)\!\in\!\mathbb{R}\!\mid\! m_1\ge 1\right\}\cap P_1,\\
    P_{1,2}&=&\left\{\!(m_1)\!\in\!\mathbb{R}\!\mid\! m_1= 0\right\}\cap P_1,\\
    P_{2,1}&=&\left\{\!(m_0,m_1)\!\in\!\mathbb{R}^2\!\mid\! m_1\ge 1\right\}\cap P_2,\\
    P_{2,2}&=&\left\{\!(m_0,m_1)\!\in\!\mathbb{R}^2\!\mid\! m_1=0\right\}\cap P_2.
  \end{eqnarray*}
  In Step~4, fortunately, no further splitting is necessary, and we can even condense
  sub cases. For $\Gamma_1,\Gamma_2$ we have the weighted representation $[n_2+1;n_2+1-m_1,1]$
  and for $\Gamma_3,\Gamma_4$ we have the weighted representation $[m_0;m_0-m_1,1]$.
\end{proof}  

Using the parametric Barvinok algorithm or elementary summation formulas with case differentiation,
we conclude the well known fact that the number of $n$-agent weighted games with type composition $(\star)$
is $n$. Similarly, we conclude, that the number of $n$-agent weighted games with type composition
$(1,\star)$ is given by
$$
  \frac{n^3-n}{6}=\frac{n(n-1)(n+1)}{6}={{n+1} \choose 3}
$$
for all $n\in\mathbb{N}$.

%

\section*{\uppercase{Appendix}}

\subsection*{Non-weighted Examples}
In order to prove Lemma~\ref{lemma_not_all_weighted} it suffices to give a non-weighted example
for each case. Those can easily be found by classifying all complete simple games with up to $7$
agents. That $(2,4)$ is non-weighted has already been observed in Section~\ref{sec:weightedness}.
For the other type compositions mentioned in Lemma~\ref{lemma_not_all_weighted}, we provide
the following examples:
\begin{itemize}
  \item $\widehat{n}=(2,2,2)$, 
        $\mathcal{S}=\begin{pmatrix}
          2 & 0 & 0\\
          1 & 2 & 0\\
          0 & 2 & 2 
        \end{pmatrix}$,
        $\tilde{L}=\begin{pmatrix}
          1 & 1 & 1
        \end{pmatrix}$,
        $x=(1,0,1)$, $y=(2)$.
  \item $\widehat{n}=(1,1,5)$, 
        $\mathcal{S}=\begin{pmatrix}
          1 & 0 & 2\\
          0 & 1 & 3
        \end{pmatrix}$,
        $\tilde{L}=\begin{pmatrix}
          1 & 1 & 0 \\
          0 & 0 & 5
        \end{pmatrix}$,
        $x=(1,1)$, $y=(1,1)$.
  \item $\widehat{n}=(1,2,3)$, 
        $\mathcal{S}=\begin{pmatrix}
          1 & 1 & 0\\
          0 & 1 & 3
        \end{pmatrix}$,
        $\tilde{L}=\begin{pmatrix}
          1 & 0 & 2 \\
          0 & 2 & 1
        \end{pmatrix}$,
        $x=(1,1)$, $y=(1,1)$.
  \item $\widehat{n}=(1,3,2)$, 
        $\mathcal{S}=\begin{pmatrix}
          1 & 0 & 2\\
          0 & 3 & 0
        \end{pmatrix}$,
        $\tilde{L}=\begin{pmatrix}
          1 & 1 & 0 \\
          0 & 2 & 2
        \end{pmatrix}$,
        $x=(1,1)$, $y=(1,1)$.
  \item $\widehat{n}=(2,1,4)$, 
        $\mathcal{S}=\begin{pmatrix}
          1 & 0 & 2\\
          0 & 1 & 3
        \end{pmatrix}$,
        $\tilde{L}=\begin{pmatrix}
          2 & 0 & 0 \\
          0 & 0 & 4
        \end{pmatrix}$,
        $x=(2,0)$, $y=(1,1)$.
  \item $\widehat{n}=(2,4,1)$, 
        $\mathcal{S}=\begin{pmatrix}
          2 & 0 & 1\\
          0 & 4 & 0
        \end{pmatrix}$,
        $\tilde{L}=\begin{pmatrix}
          1 & 2 & 1
        \end{pmatrix}$,
        $x=(1,1)$, $y=(2)$.                                        
   \item $\widehat{n}=(1,1,1,3)$, 
        $\mathcal{S}=\begin{pmatrix}
          1 & 0 & 0 & 2\\
          0 & 1 & 1 & 1\\
          0 & 1 & 0 & 3
        \end{pmatrix}$,
        $\tilde{L}=\begin{pmatrix}
          1 & 1 & 0 & 0\\
          0 & 0 & 1 & 3
        \end{pmatrix}$,
        $x=(1,1,0)$, $y=(1,1)$.        
  \item $\widehat{n}=(1,1,3,1)$, 
        $\mathcal{S}=\begin{pmatrix}
          1 & 0 & 1 & 1\\
          0 & 1 & 2 & 0   
        \end{pmatrix}$,
        $\tilde{L}=\begin{pmatrix}
          1 & 1 & 0 & 0\\
          0 & 0 & 3 & 1  
        \end{pmatrix}$,
        $x=(1,1)$, $y=(1,1)$.                
  \item $\widehat{n}=(2,1,2,1)$, 
        $\mathcal{S}=\begin{pmatrix}
          1 & 1 & 0 & 1\\
          1 & 0 & 2 & 0
        \end{pmatrix}$,
        $\tilde{L}=\begin{pmatrix}
          2 & 0 & 0 & 0\\
          0 & 1 & 2 & 1
        \end{pmatrix}$,
        $x=(1,1)$, $y=(1,1)$.
   \item $\widehat{n}=(2,3,1,1)$, 
        $\mathcal{S}=\begin{pmatrix}
          1 & 1 & 0 & 1\\
          0 & 3 & 1 & 0   
        \end{pmatrix}$,
        $\tilde{L}=\begin{pmatrix}
          2 & 0 & 0 & 0\\
          1 & 0 & 1 & 1\\
          0 & 3 & 0 & 1  
        \end{pmatrix}$. The vectors
        $x=(3)$, $y=(1,1,1)$ yield $(3,3,0,3)\preceq(3,3,1,2)$,
        which is generally also sufficient to conclude non-weightedness, since
        the coalition vectors can be shifted accordingly. Here we may consider
        two times the SMW vector $(1,1,0,1)$ and one times the shifted winning
         vector $(1,0,1,1)$ to obtain $(3,3,1,2)$.  
   \item $\widehat{n}=(1,2,2,1)$, 
        $\mathcal{S}=\begin{pmatrix}
          1 & 0 & 1 & 1\\
          0 & 2 & 1 & 0   
        \end{pmatrix}$,
        $\tilde{L}=\begin{pmatrix}
          1 & 1 & 0 & 0\\
          0 & 1 & 2 & 1  
        \end{pmatrix}$,
        $x=(1,1)$, $y=(1,1)$.
  \item $\widehat{n}=(1,2,1,2)$, 
        $\mathcal{S}=\begin{pmatrix}
          1 & 0 & 0 & 2\\
          0 & 2 & 1 & 0\\
          0 & 2 & 0 & 2   
        \end{pmatrix}$,
        $\tilde{L}=\begin{pmatrix}
          1 & 1 & 0 & 0\\
          0 & 1 & 1 & 2  
        \end{pmatrix}$,
        $x=(1,1,0)$, $y=(1,1)$.
  \item $\widehat{n}=(1,1,2,2)$, 
        $\mathcal{S}=\begin{pmatrix}
          1 & 0 & 0 & 2\\
          0 & 1 & 1 & 1   
        \end{pmatrix}$,
        $\tilde{L}=\begin{pmatrix}
          1 & 1 & 0 & 0\\
          0 & 1 & 0 & 2\\
          0 & 0 & 2 & 2  
        \end{pmatrix}$,
        $x=(1,2)$, $y=(1,1,1)$.
  \item $\widehat{n}=(1, 2, 1, 1, 1)$, 
        $\mathcal{S}=\begin{pmatrix}
          1 & 1 & 0 & 0 & 0\\
          1 & 0 & 1 & 1 & 0\\
          0 & 2 & 1 & 0 & 0\\
          0 & 1 & 1 & 1 & 1       
        \end{pmatrix}$,
        $\tilde{L}=\begin{pmatrix}
          1 & 0 & 1 & 0 & 1\\
          0 & 2 & 0 & 1 & 0)    
        \end{pmatrix}$,
        $x=(1,0,0,1)$, $y=(1,1)$.        
  \item $\widehat{n}=(1, 1, 2, 1, 1)$, 
        $\mathcal{S}=\begin{pmatrix}
          1 & 1 & 0 & 0 & 0\\
          1 & 0 & 1 & 1 & 0\\
          0 & 1 & 2 & 0 & 0\\
          0 & 0 & 2 & 1 & 1     
        \end{pmatrix}$,
        $\tilde{L}=\begin{pmatrix}
          1 & 0 & 1 & 0 & 1\\
          0 & 1 & 1 & 1 & 0
        \end{pmatrix}$,
        $x=(1,0,0,1)$, $y=(1,1)$.
  \item $\widehat{n}=(1, 1, 1, 2, 1)$, 
        $\mathcal{S}=\begin{pmatrix}
          1 & 1 & 0 & 0 & 0\\
          1 & 0 & 1 & 0 & 1\\
          1 & 0 & 0 & 2 & 0\\
          0 & 0 & 1 & 2 & 1             
        \end{pmatrix}$,
        $\tilde{L}=\begin{pmatrix}
          0 & 1 & 1 & 1 & 0\\
          1 & 0 & 0 & 1 & 1  
        \end{pmatrix}$,
        $x=(1,0,0,1)$, $y=(1,1)$.        
  \item $\widehat{n}=(1, 1, 1, 1, 2)$, 
        $\mathcal{S}=\begin{pmatrix}
          1 & 1 & 0 & 0 & 0\\
          1 & 0 & 1 & 0 & 1\\
          0 & 1 & 1 & 1 & 0\\
          0 & 0 & 1 & 1 & 2             
        \end{pmatrix}$,
        $\tilde{L}=\begin{pmatrix}
          1 & 0 & 0 & 1 & 1\\
          0 & 1 & 1 & 0 & 1    
        \end{pmatrix}$,
        $x=(1,0,0,1)$, $y=(1,1)$.      
\end{itemize}  

\subsection*{Proofs of some Special Cases of Conjecture~\ref{main_conj}}
For $n\le 5$ agents, all complete simple games are weighted, so that all type compositions
$\widehat{n}$ with $\Vert\widehat{n}\Vert_1\le 5$ are weighted, i.e., $(1,2,2)$ and $(1,1,2,1)$
are weighted. For $n=6$ agents, we remark that exactly $1111$ of the $1171$ complete simple games
are weighted, i.e., $60$ games are non-weighted. For $n\ge 7$ agents, the discrepancy quickly
increases. For $n=7$ agents there are $44313$ complete simple games, while only $29373$ of them
are weighted. For $n=8$ agents, the respective counts are given by $16175188$ and $2730164$. Via
exhaustive enumeration one can easily show that $(1,1,4)$ is weighted. The cases $(\star)$ and
$(\star,1)$ have already been proven in Subsection~\ref{subsec:example_for_algorithm}.

\begin{lemma}
  \label{lemma_star_1}
  $\omega=(\star,1)$ is weighted.
\end{lemma}
\begin{proof}
  According to Step~1 we have $\tau(\omega)=1$ parameter $m_0$, $\Lambda=1$, and $C(\omega)=\{(m_0,0),(m_0,1),
  (m_0-1,0),(m_0-1,1)\}$ with $|C(\omega)|=4$. For Step~2 we remark that no pair of elements of $C(\omega)$
  is incomparable, so that we have $r=1$. There is no need to consider the cases $\mathcal{S}=(m_0-1,0)$
  or $\mathcal{S}=(m_0-1,1)$ separately, since they are captured by the two cases $\mathcal{S}=(m_0,0)$
  and $\mathcal{S}=(m_0,1)$. Generally, we can require that for each used parameter $m_j$ there is a SMW
  vector taking the value $m_j-0$ at the respective coordinate. Due to condition (a)(iii) of 
  Theorem~\ref{thm_characterization_cs}, $\mathcal{S}=(m_0,1)$ is not possible. Thus, there
  remains only one case with non-empty polytope for the parameter:
  $$
    \mathcal{S}_1=\begin{pmatrix}m_0 & 0\end{pmatrix}, 
    P_1=\left\{(m_0)\in\mathbb{R}\mid 1\le m_0\le n_1\right\}.
  $$  
  In Step~3 there remains a single case: $\Gamma_1=(\mathcal{S}_1,\mathcal{L}_1,P_1)$
  with $\mathcal{L}=\begin{pmatrix}m_0-1&0\end{pmatrix}$.
  In Step~4 we determine the weighted representation $[m_0;1,0]$ for $\Gamma_1$.
\end{proof}

\begin{lemma}
  $\omega=(\star,2)$ is weighted.
\end{lemma}
\begin{proof}
  According to Step~1 we have $\tau(\omega)=1$ parameter $m_0$, $\Lambda=2$, and $C(\omega)=\{(m_0,0),(m_0,1),
  (m_0,2),(m_0-1,0),(m_0-1,1),(m_0-1,2),(m_0-2,0),(m_0-2,1)$, $(m_0-2,2)\}$ with $|C(\omega)|=9$. As in the proof
  of Lemma~\ref{lemma_star_1}, we remark that we can assume that one of the vectors $(m_0,0)$, $(m_0,1)$, or
  $(m_0,2)$ has to be chosen as a SMW vector. Only the first one is incomparable with one of the other
  vectors of $C(\omega)$. $\mathcal{S}=\begin{pmatrix}m_0 &2\end{pmatrix}$ is impossible due to condition
  (a)(iii) of Theorem~\ref{thm_characterization_cs}. Thus Step~2 yields the cases
  \begin{itemize}
    \item $\mathcal{S}_1=\begin{pmatrix}m_0&0\end{pmatrix}$,
          $P_1=\left\{(m_0)\in\mathbb{R}\mid 1\le m_0\le n_1\right\}$
    \item $\mathcal{S}_2=\begin{pmatrix}m_0&1\end{pmatrix}$,
          $P_2=\left\{(m_0)\in\mathbb{R}\mid 1\le m_0\le n_1\right\}$ 
    \item $\mathcal{S}_3=\begin{pmatrix}m_0&0\\m_0-1&2\end{pmatrix}$,
          $P_3=P_2=P_1$
  \end{itemize}  
  No splitting or modification of the $P_i$ is necessary in Step~3, so that
  $\Gamma_i=(\mathcal{S}_i,\mathcal{L}_i,P_i)$, where $\mathcal{L}_1=\begin{pmatrix}m_0-1&2\end{pmatrix}$,
  $\mathcal{L}_2=\begin{pmatrix}m_0&0\\m_0-1&2\end{pmatrix}$, and $\mathcal{L}_3=\begin{pmatrix}m_0-1&1\end{pmatrix}$
  for all $1\le i\le 3$. In Step~4 we determine the weighted representations $[m_0;1,0]$ for $\Gamma_1$,
  $[2m_0+1;2,1]$ for $\Gamma_2$, and $[2m_0;2,1]$ for $\Gamma_3$.
\end{proof}

\begin{lemma}
  $\omega=(\star,3)$ is weighted.
\end{lemma}
\begin{proof}
  According to Step~1, we have $\tau(\omega)=1$ parameter $m_0$, $\Lambda=3$, so that $|C(\omega)|=16$. As in
  the two previous proofs, we remark that we can assume that one of the vectors $(m_0,0)$, $(m_0,1)$, $(m_0,2)$,
  or $(m_0,3)$ has to be chosen as a SMW vector. Here we can check that at most $r=2$ elements of $C(\omega)$ can
  be pairwise incomparable. Excluding $\mathcal{S}=\begin{pmatrix}m_0 & 3\end{pmatrix}$ with condition (a)(iii)
  of Theorem~\ref{thm_characterization_cs}, there remain the following possibilities for $r=1$:
  \begin{itemize}
    \item $\mathcal{S}_1=\begin{pmatrix}m_0&0\end{pmatrix}$,
          $P_1=\left\{(m_0)\in\mathbb{R}\mid 1\le m_0\le n_1\right\}$
    \item $\mathcal{S}_2=\begin{pmatrix}m_0&1\end{pmatrix}$,
          $P_2=P_1$ 
    \item $\mathcal{S}_3=\begin{pmatrix}m_0&2\end{pmatrix}$,
          $P_3=P_1$
  \end{itemize}
  For $r=2$ we additionally obtain:
  \begin{itemize}
    \item $\mathcal{S}_4=\begin{pmatrix}m_0&0\\m_0-1&2\end{pmatrix}$,
          $P_4=P_1$
    \item $\mathcal{S}_5=\begin{pmatrix}m_0&0\\m_0-1&3\end{pmatrix}$,
          $P_5=P_1$
    \item $\mathcal{S}_6=\begin{pmatrix}m_0&1\\m_0-1&3\end{pmatrix}$,
          $P_6=P_1$
    \item $\mathcal{S}_7=\begin{pmatrix}m_0&0\\m_0-2&3\end{pmatrix}$,
          $P_7=[2,n_1]$                    
  \end{itemize}
  For $i\in\{1,2,5,6,7\}$ no splitting of the $P_i$ is necessary in Step~3, so that
  $\Gamma_i=(\mathcal{S}_i,\mathcal{L}_i,P_i)$, where
  $\mathcal{L}_1=\begin{pmatrix}m_0-1 & 3\end{pmatrix}$, 
  $\mathcal{L}_2=\begin{pmatrix}m_0&0\\m_0-1&3\end{pmatrix}$,
  $\mathcal{L}_5=\begin{pmatrix}m_0-1&2\end{pmatrix}$,
  $\mathcal{L}_6=\begin{pmatrix}m_0&0\\m_0-1&2\end{pmatrix}$,
  $\mathcal{L}_7=\begin{pmatrix}m_0-1&1\end{pmatrix}$. For the remaining
  indices $i\in\{3,4\}$ cases split up as $\Gamma_{i,j}=(\mathcal{S}_i,\mathcal{L}_{i,j},P_{i,j})$,
  where
  \begin{itemize}
    \item $\mathcal{L}_{3,1}=\begin{pmatrix}m_0&1\\m_0-1&3\end{pmatrix}$, $P_{3,1}=\{n_1\}$
    \item $\mathcal{L}_{3,2}=\begin{pmatrix}m_0+1&0\\m_0-1&3\end{pmatrix}$, $P_{3,2}=[1,n_1-1]$
    \item $\mathcal{L}_{4,1}=\begin{pmatrix}m_0-1&1\end{pmatrix}$, $P_{3,1}=\{1\}$
    \item $\mathcal{L}_{4,2}=\begin{pmatrix}m_0-1&1\\m_0-2&3\end{pmatrix}$, $P_{3,2}=[2,n_1]$
  \end{itemize}
  In Step~4 we obtain the weighted representations
  \begin{itemize}
    \item $[m_0;1,0]$ for $\Gamma_1$
    \item $[3m_0+1;3,1]$ for $\Gamma_2$
    \item $[3m_0+4;3,2]$ for $\Gamma_{3,1}$ and $\Gamma_{3,2}$ 
    \item $[2m_0;2,1]$ for $\Gamma_{4,1}$ and $\Gamma_{4,2}$
    \item $[3m_0;3,1]$ for $\Gamma_5$
    \item $[2m_0+1;2,1]$ for $\Gamma_6$
    \item $[3m_0;3,2]$ for $\Gamma_7$
  \end{itemize}  
\end{proof}

Thus, Conjecture~\ref{main_conj} and Corollary~\ref{main_cor} are proven for all type compositions
with $t$ types, where either $t<3$ or $t>5$. For $t\in\{3,4,5\}$, the necessary case differentiations
become more and more complex so that we aim for a computer aided proof. 

\subsection*{Open Problems}
Of course, Conjecture~\ref{main_conj} has to be proven in the first run. Several ways to extend these
considerations are imaginable. One can ask for a similar classification for roughly-weighted games,
i.e., for games where the weight of a winning coalition can equal the weight of another losing coalition,
or for complete simple games of small dimension, i.e., games which can be represented as the intersection
of a small number of weighted games. A possibly more challenging open problem is to provide explicit
enumeration results for other sub classes of weighted games, than the ones known in the literature or
those that can be concluded from the line of consideration presented in this paper.

\end{document}